\documentclass{amsart}
\usepackage{amssymb}
\usepackage{graphics}

\theoremstyle{plain}

\newtheorem*{theorem}{Theorem}
\newtheorem*{corollary}{Corollary}

\def    \R  {{\Bbb R}}
\def    \Z  {{\Bbb Z}}

\def    \CP {{\Bbb {CP}}}

\def    \C  {{\Bbb C}}

\begin{document}
\title[Spherical contact toric manifolds]{Spherical contact toric manifolds}

\author{Hui Li} 
\address{School of mathematical Sciences\\
        Soochow University\\
        Suzhou, 215006, China.}
        \email{hui.li@suda.edu.cn}

\thanks{2010 classification. Primary:  53D10, 53D20; Secondary:
55N10, 57R20}

\begin{abstract}
Let $(M, \alpha)$ be a $2n+1$-dimensional connected compact contact toric manifold of Reeb type.  Suppose the contact form $\alpha$ is regular, we find conditions under which $M$ is homeomorphic to $S^{2n+1}$.
\end{abstract}

 \maketitle
Let $(M, \alpha)$ be a contact manifold of dimension $2n+1$. If an $n+1$-dimensional torus $T^{n+1}$ acts on $M$ effectively, preserving the contact form $\alpha$, then $M$ is called a {\bf contact toric manifold}.
 If the Reeb vector field of the contact form $\alpha$ is generated by a one parameter subgroup of $T^{n+1}$, then the $T^{n+1}$-manifold $M$ is called a {\bf contact toric manifold of Reeb type}. Compact contact toric manifolds are classified by Lerman \cite{Lc}.  Most of the compact contact toric manifolds are of Reeb type. 

 Suppose $(M, \alpha)$ is a connected compact contact toric manifold of dimension $2n+1$. Let $\mathfrak t$ be the Lie algebra of the torus $T^{n+1}$, and $\mathfrak t^*$  the dual Lie algebra. The {\bf contact moment map}  $\phi\colon M\to\mathfrak t^*$ is defined to be
$$\langle\phi(x), X\rangle = \alpha_x\big(X_M(x)\big), \,\, \forall\, x\in M,\, \mbox{and}\,\,\, \forall \,\, X\in\mathfrak t,$$
 where $X_M$ is the vector field on $M$ generated by the $X$-action.
The {\bf moment cone} of $\phi$ is defined to be
$$C(\phi) =\big\{t\phi(x)\,|\, t\geq 0,\, x\in M\big\}.$$
It is known (\cite{{BG}, {Lc}} etc.) that, if $M$ is of Reeb type, 
then the moment map image is an $n$-dimensional convex polytope lying in a hyperplane away from the origin, and  $C(\phi)$ is a cone over the polytope, it is a {\it strictly convex rational polyhedral cone} of dimension $n+1$. {\it Strictly convex} means that $C(\phi)$ contains no linear subspaces of $\mathfrak t^*$ of positive dimension, {\it polyhedral} means that $C(\phi)$ is a cone over a polytope, and {\it rational} means that the normal vectors of the facets of the cone  lie in the integral lattice of $\mathfrak t$.  

Let $(M, \alpha)$ be a contact manifold. The contact form $\alpha$ is called 
{\bf regular} if the Reeb flow of $\alpha$ generates a {\it free} circle action.
The sphere $S^{2n+1}$ with its standard contact form $\alpha_0$ (the contact form induced from the standard symplectic form on $\C^{n+1}$) is a contact manifold with a regular contact form (see \cite[Example 4.9]{BG} for example). The sphere $(S^{2n+1}, \alpha_0)$ is a contact toric manifold of Reeb type, the $T^{n+1}$ action is the one induced from the standard $T^{n+1}$ action on $\C^{n+1}$.
The moment map image of $(S^{2n+1}, \alpha_0)$ is the standard simplex in $\R^{n+1}$, and the moment cone  is the cone over the standard simplex, which is the nonnegative coordinate zone.  For a contact toric manifold $(M, \alpha)$ of Reeb type, we can always find a suitable contact form $\alpha'$, which defines the same contact structure as $\alpha$, such that its Reeb flow generates a locally free circle action (see \cite{BG}).

In this paper, we study compact  contact toric manifolds of Reeb type with regular contact forms, we give some sufficient conditions under which the manifolds are homeomorphic to $S^{2n+1}$. 

\begin{theorem}
Let $(M,  \alpha)$ be a $2n+1$-dimensional connected compact contact toric manifold of Reeb type. Suppose the contact form is regular so that the Reeb flow generates a {\it free} $S^1$ action. 
If the orbit space $N=M/S^1$ of the Reeb flow  satisfies any one of the following two conditions, then $M$ is homeomorphic to  $S^{2n+1}$.
\begin{enumerate}
\item $H^*(N; \Z) = \Z[x]/x^{n+1}$, where $x=[\Omega]$, with $\pi\colon M\to N$ and $\pi^*\Omega = d\alpha$.
\item $c_1(N) = (n+1)[\Omega]$.  
\end{enumerate} 
\end{theorem}

\begin{proof}[Proof of Theorem]
Since the Reeb flow corresponds to a free $S^1$-action,  $N=M/S^1$ is a smooth manifold; it is a compact symplectic toric manifold of dimension $2n$. Since
the fixed points of the $T^n$-action on $N$ are isolated, we may choose a subcircle action of the $T^n$-action on $N$ such that its moment map  has an isolated point as its minimum, by \cite{L0}, $N$ is simply connected, i.e., $$\pi_1(N)=1.$$

$\pi\colon M\to N$ is a principal $S^1$-bundle. By the theorems of Boothby and Wang \cite[Theorem 1]{BW}, and Kobayashi \cite{K},
$\alpha$ defines a connection form on $M$,
$\pi^*\Omega$ is the curvature form of this connection, $[\Omega]$ is an {\it integral} class, and it is the first Chern class of this circle bundle.

\medskip

First, suppose $N$ satisfies condition $(1)$, i.e., the integral cohomology ring of $N$ is 
$$H^*(N; \Z) = \Z[x]/x^{n+1}, \,\,\,\mbox{where $x=[\Omega]$}.$$
This first implies that $[\Omega]\in H^2(N; \Z)$ is a generator. 
We look at the homotopy exact sequence of the $S^1$-bundle $\pi\colon M\to N$
$$\cdots\to \pi_2(N)\to\pi_1(S^1)\to\pi_1(M)\to\pi_1(N)\to\cdots.$$
By considering the universal $S^1$-bundle $ES^1\to BS^1$, a classifying map $N\to BS^1$, and a diagram chasing, we see that the map 
$$\pi_2(N)= H_2(N; \Z) = H^2(N; \Z)=\Z\to\pi_1(S^1) = H_1(S^1;\Z) = H^1(S^1; \Z)=\Z$$
is multiplication by the first Chern class $[\Omega]\in H^2(N; \Z)=\Z$ of the $S^1$-bundle $\pi\colon M\to N$
(we need $\pi_1(N) = 1$ for the claim $\pi_2(N)= H_2(N; \Z)$). Since $[\Omega]$ is primitive, (up to a sign) the above map is multiplication by $1$, so the map
$$\pi_2(N)\to\pi_1(S^1)$$
 is surjective; since we also have $\pi_1(N) = 1$, by the above exact sequence, we get
$$\pi_1(M) = 1.$$
We now look at the Gysin exact sequence for the circle bundle $\pi\colon M\to N$, 
$$\cdots\to H^*(M; \Z)\to H^{*-1}(N; \Z)\to H^{*+1}(N; \Z)\to H^{*+1}(M; \Z)\to\cdots,$$
by the above integral cohomology ring structure of $N$, we know that the odd cohomology groups of $N$ are $0$ and the map $H^{*-1}(N; \Z)\to H^{*+1}(N; \Z)$ is an isomorphism for each $1\leq *\leq 2n-1$, so we get 
$$H^*(M; \Z) = \Z \,\,\,\mbox{when $*=0, 2n+1$},\,\,\,\mbox{and $H^*(M; \Z) = 0$ when $*\neq 0, 2n+1$}.$$ 
We proved that $M$ has the integral cohomology ring of $S^{2n+1}$.

Since $M$ is simply connected, and has the integral cohomology ring of $S^{2n+1}$, by the generalized Poincar\'e conjecture \cite{KL, S}, $M$ is homeomorphic to $S^{2n+1}$. (For dimension $3$, by the classification in \cite{Lc} for compact connected contact toric manifolds of Reeb type, the manifolds are diffeomorphic to the lens spaces. Since $M$ is simply connected, $M$ is diffeomorphic to $S^3$. This is an alternative argument for dimension $3$.)

\medskip

Next  suppose $N$ satisfies condition $(2)$, i.e., 
$c_1(N) = (n+1)[\Omega]$. 
The manifold $(N, \Omega)$, as a symplectic toric manifold, by Delzant's theorem (\cite{D}), is equivariantly symplectomorphic to the symplectic reduced space of the K\"ahler manifold $\C^m$ (for some $m > 0$) by a torus action which preserves the K\"ahler structure of $\C^m$. By \cite[Theorem 3.5]{GS}, this symplectic reduced space is a K\"ahler manifold. Hence $(N, \Omega)$ possesses a K\"ahler structure compatible with $\Omega$. Since $N$ is K\"ahler, and $c_1(N) = (n+1)[\Omega]$ with $[\Omega]$ being an integral K\"ahler class, by a theorem of Kobayashi and Ochiai  (\cite{KO}), 
$N$ is biholomorphic to $\CP^n$. The class $[\Omega]$ must be primitive, since if it was not,
then $c_1(N) \neq c_1(\CP^n)$, contradicting to $N$ being biholomorphic to $\CP^n$. Then $N$ has the integral cohomology ring of $\CP^n$, with $[\Omega]$ being the degee 2 generator; the rest of the proof then may go as in the last case.  Or alternatively, once we know that $N$ is biholomorphic to $\CP^n$, and the first Chern class of the $S^1$-bundle $M\to N$ is $[\Omega]$, with $[\Omega]$ being primitive integral, then as in the last case, we get that $\pi_1(M) = 1$; by the classification of Kobayashi (\cite[Theorem 8]{K}), when the total space is simply connected, an $S^1$-bundle over $\CP^n$ is classified by its first Chern class, hence $M$ is
homeomorphic to $S^{2n+1}$.
\end{proof}

\begin{corollary}
Let $(M,  \alpha)$ be a $2n+1$-dimensional connected compact contact toric manifold of Reeb type. Suppose the contact form $\alpha$ is regular so that the Reeb flow generates a {\it free} $S^1$ action. If the orbit space $N= M/S^1$ satisfies any one of the two conditions in the theorem, then the moment cone of $M$ has exactly $n+1$ number of facets, and the primitive
integral normal vectors of the facets span a $\Z$ basis of $\Z^{n+1}$.
\end{corollary}

\begin{proof}
When the orbit space $N$ satisfies any one of the two conditions of our theorem, by the proof, $N$ has the same Betti numbers as $\CP^n$, i.e.,  $b_{2i}(N) = 1$  and $b_{2i+1}(N) = 0$ for any $0\leq 2i\leq 2n$. 

The symplectic manifold $(N, \Omega)$
has an effective $T^n = T^{n+1}/S^1$ action. Let $\bar\phi$ be the moment map of the $T^n$-action on $N$. The fixed points of the $T^n$ action on $N$ are isolated, and the $n$-dimensional moment polytope $\bar\phi(N)$ is the convex hull of the images of the $T^n$ fixed points  (\cite{A, GS1}).
Since $d\alpha = \pi^*\Omega$, the polytopes $\bar\phi(N)$ and $\phi(M)$ differ by a constant in $\mathfrak t^*$.

We choose a generic element $\eta\in\mathfrak t' =$ Lie$(T^n)$. Then the moment map component $\bar\phi^{\eta}$ is a perfect Morse function on $N$. Together with the fact on the Betti numbers of $N$, we know that $\bar\phi^{\eta}$ has $n+1$ number of critical points, which are precisely the fixed points of the one parameter subgroup $\exp(t\eta)$-action, which are precisely the same as the fixed points of the $T^n$-action on $N$. So the moment polytope $\bar\phi(N)$ has $n+1$ number of vertices, and then it has $n+1$ number of facets. Then $\phi(M)$ has
$n+1$ number of facets, and the moment cone $C(\phi)$ has $n+1$ number of facets.

By our theorem, the assumptions imply that $\pi_1(M) = 1$.  Let $v_1, \cdots, v_{n+1}$ be  the primitive integral normal vectors of the facets of the moment cone, by \cite{L}, up to a sign,
$$\pi_1(M) = \det [v_1, \cdots, v_{n+1}].$$
Hence $\{v_1, \cdots, v_{n+1}\}$ forms a $\Z$-basis of $\Z^{n+1}$.
\end{proof}

 A  Riemannian metric $g$ on a contact manifold $(M, \alpha)$ is said to be {\it adapted} to $\alpha$ if $g$ is preserved by the Reeb flow of $\alpha$, and  if there exists an almost complex structure $J$ on  $\ker(\alpha)$ such that $g$, $J$ and $d\alpha$ are compatible on the symplectic subbundle 
$\ker(\alpha)$. In such a case, we call $(\alpha, g)$ a K-contact structure on $M$, or we call $(M, \alpha, g)$ a {\bf K-contact manifold}. A compact K-contact manifold $(M, \alpha, g)$ admits a torus action preserving $\alpha$, where the torus is the closure of the Reeb flow of $\alpha$ in the isometry group of $M$ (which is compact by 
\cite{MS}). If the compact K-contact manifold $(M, \alpha, g)$ is of dimension $2n+1$, then the dimension of the torus is bounded from above  by $n+1$ (\cite{R}). 
If $(M, \alpha)$ is a contact toric manifold of Reeb type, then we can choose an invariant Riemannian metric $g$ adapted to $\alpha$ such that $M$ is given a K-contact structure $(\alpha, g)$, hence is a K-contact manifold (see for example \cite{Lh}). So compact contact toric manifolds of Reeb type belong to the category of K-contact manifolds.

The sphere $(S^{2n+1}, \alpha_0)$ with the standard contact form $\alpha_0$ admits a (standard) K-contact structure. A perturbation of $\alpha_0$ can give us a different K-contact  structure on $S^{2n+1}$ such that the Reeb flow has exactly $n+1$ number of closed orbits (which are circles) \cite{R2}, this is the minimal number of closed Reeb orbits that a $2n+1$-dimensional compact K-contact manifold can have.

Let $(M, \alpha)$ be a $2n+1$-dimensional compact connected contact toric manifold of Reeb type with a regular contact form. Under the condition of our theorem, by the corollary, the moment cone of $(M, \alpha)$ has exactly $n+1$ number of facets, or the moment polytope $\phi(M)$ has $n+1$ number of facets. We identify the moment polytope of the $T^n=T^{n+1}/S^1$-action on $N=M/S^1$ with $\phi(M)$, and identify the Lie algebra of $T^n$ with a subalgebra
of $\mathfrak t$. We choose a generic element $\eta\in\mathfrak t' =$ Lie$(T^n)$, then $X = (\eta, \xi)$ is a generic element in $\mathfrak t$, where $\xi$ is the generator of the Lie algebra of the subgroup $S^1$. Let $X_M$ be the  generating vector field of  $X = (\eta, \xi)$ in $M$. Then the flow of $X_M$ has exactly $n+1$ number of closed orbits. By the choice of $X$ and by \cite[Theorem 7.2]{YK},  (may also see the proof of \cite[Prop. 1]{R2}), we can find a contact form $\alpha'$ on $M$ whose Reeb vector field is $X_M$. Then we have a (different) K-contact structure $(\alpha', g')$ on $M$ with exactly $n+1$ number of closed Reeb orbits. Hence the manifold we study in our theorem is a compact connected manifold of dimension $2n+1$ with a K-contact structure which has exactly $n+1$ number of closed Reeb orbits (minimal number of closed Reeb orbits).
Compact connected K-contact  manifolds of dimension $2n+1$ with exactly $n+1$ number of closed Reeb orbits are studied in \cite{R1, R2, R3}, where the theorems claim that 
such manifolds are finitely covered by $S^{2n+1}$, are homeomorphic to $S^{2n+1}$ when they are simply connected. However, a counter example of this claim is given in \cite{GNT}, the $7$-dimensional Stiefel manifold $V_2(\R^5)$, which is simply connected and is not homeomorphic to $S^7$.

\end{document}